\documentclass[10pt,final,twocolumn]{IEEEtran}                                                          % if you need a4paper
\newtheorem{proposition}{Proposition}
\newtheorem{theorem}{Theorem}
\newtheorem{definition}{Definition}

\usepackage{cite}

\ifCLASSINFOpdf
\usepackage[pdftex]{graphicx}
\graphicspath{{../pdf/}{../jpeg/}}
\DeclareGraphicsExtensions{.pdf,.jpeg,.png}
\else
\usepackage[dvipdf]{graphicx}
\graphicspath{{../eps/}}
\DeclareGraphicsExtensions{.eps}
\fi

\usepackage[cmex10]{amsmath}
\usepackage{amssymb}
\usepackage{array}
\usepackage{mdwmath}
\usepackage{mdwtab}
\usepackage{eqparbox}
\usepackage[tight,footnotesize]{subfigure}

\title{\LARGE \bf
Time-Optimal Adiabatic-Like Expansion of Bose-Einstein Condensates
}

\author{Dionisis Stefanatos and Jr-Shin Li% <-this % stops a space
\thanks{This work was supported by the Air Force Office of Scientific Research under Young Investigator Award FA9550-10-1-0146}% <-this % stops a space
\thanks{D. Stefanatos and J.-S. Li are with the Department of Electrical and Systems Engineering, Washington University,
        St. Louis, MO 63130, USA
        {\tt\small dionisis@seas.wustl.edu, jsli@seas.wustl.edu}}%
}
\begin{document}

\maketitle
%\thispagestyle{empty}
%\pagestyle{empty}

%%%%%%%%%%%%%%%%%%%%%%%%%%%%%%%%%%%%%%%%%%%%%%%%%%%%%%%%%%%%%%%%%%%%%%%%%%%%%%%%
\begin{abstract}

In this paper we study the fast adiabatic-like expansion of a one-dimensional Bose-Einstein condensate (BEC) confined in a harmonic potential, using the theory of time-optimal control. We find that under reasonable assumptions suggested by the experimental setup, the minimum-time expansion occurs when the frequency of the potential changes in a bang-bang form between the permitted values. We calculate the necessary expansion time and show that it scales logarithmically with large values of the expansion factor. This work is expected to find applications in areas where the efficient manipulations of BEC is of utmost importance. As an example we present the field of atom interferometry with BEC, where the wavelike properties of atoms are used to perform interference experiments that measure with unprecedented precision small shifts induced by phenomena like rotation, acceleration, and gravity gradients.

\end{abstract}


\IEEEpeerreviewmaketitle

%%%%%%%%%%%%%%%%%%%%%%%%%%%%%%%%%%%%%%%%%%%%%%%%%%%%%%%%%%%%%%%%%%%%%%%%%%%%%%%%
\section{INTRODUCTION}

A Bose-Einstein condensate (BEC) is the state of matter emerging in a dilute gas of weakly interacting bosonic atoms confined in an external potential when cooled to temperatures very near to absolute zero. More specifically, below a critical temperature a large fraction of the bosons occupies the lowest quantum state of the external potential, and the quantum effects become apparent on a macroscopic scale. This peculiar state of matter was first predicted by Einstein, when he extended the statistics of light quanta (photons), developed by Bose \cite{Bose24}, to material particles \cite{Einstein25}.

Since its first experimental demonstration in 1995 \cite{Cornell95, Ketterle95}, BEC has become a workhorse for atomic physics experiments. One premier example is the use of BEC for precision measurements in the context of atom interferometry \cite{Cronin09,Schmiedmayer11}, where the wavelike properties of atoms are exploited to perform interference experiments that measure small shifts induced by phenomena like rotation, acceleration, and gravity gradients \cite{Kasevich11}. For this kind of applications, the necessity to control and manipulate the BEC is ubiquitous.

Among the various necessary control steps, a very important task is the ability to expand the BEC without exciting higher states, an undesirable effect that generates friction and heating \cite{Rezek06}. The conventional way to achieve this goal is to change adiabatically the potential that confines the BEC. The drawback of this
approach is the long necessary times which may render it impractical. Recently, a method to bypass this problem has been proposed \cite{Muga09} and implemented experimentally \cite{Schaff_EPL11}. The idea is to change the trapping potential in a way that prepares the same
final state as the adiabatic process at a given final time. This method, neatly characterized as ``shortcut to adiabaticity" \cite{Chen10}, provides a family of paths which achieve the desired frictionless expansion, and in theory the necessary time can be made arbitrarily small. In practice, there are always experimental constraints which limit this time to a finite value.

In this article, we impose restrictions motivated by experiments and formulate the desired transfer as a time-optimal problem. We then use optimal control theory to
find the shortest adiabatic-like path for a one-dimensional BEC, trapped in a time-dependent harmonic potential. We show that the minimum time frictionless expansion takes place when the frequency of the potential changes in a bang-bang manner between the allowed values and calculate the necessary expansion time. Note that numerical optimization methods have been used to control a BEC in an optical lattice \cite{Sklarz02} or in a magnetic microtrap \cite{Hohenester07}. Our approach here is different, since we use the time-optimal theory of single-input control systems in the plane \cite{Su1,Boscain04}. The analysis continuous our previous work \cite{Stefanatos10,Stefanatos11}, where we have considered the minimum time frictionless cooling of a harmonically trapped atom.

\section{FORMULATION OF THE PROBLEM IN TERMS OF OPTIMAL CONTROL}

The evolution of the wavefunction $\psi(t,x)$ of a condensate trapped
in a one-dimensional (elongated cigar trap) parabolic potential with time-varying frequency
$\omega(t)$ is given by the following Gross-Pitaevskii equation \cite{Muga09}
\begin{equation}
\label{Schrodinger}i\hbar\frac{\partial\psi}{\partial t}=\left[  -\frac
{\hbar^{2}}{2m}\frac{\partial^{2}}{\partial x^{2}}+\frac{m\omega^{2}%
(t)}{2}x^{2}+g|\psi|^2\right]  \psi,
\end{equation}
where $m$ is the particle mass, $g$ is the coupling constant and $\hbar$ is Planck's constant; $x\in\mathbb{R}$ and $\psi$ is a square-integrable function on the real line. When
$\omega(t)$ is \emph{constant} and $g|\psi|^2/(\hbar\omega)\gg 1$ the kinetic energy term may be neglected \cite{Baym96}, the so-called Thomas-Fermi approximation,
\begin{equation}
\label{TF1}i\hbar\frac{\partial\psi}{\partial t}=\left[  \frac{m\omega^{2}}{2}x^{2}+g|\psi|^2\right]  \psi.
\end{equation}
The above equation can be solved by separation
of variables and the solution is
\begin{equation}
\label{solution_constant}\psi(t,x)=e^{-i\mu_{\omega
}t/\hbar}\,\Psi_{\omega}(x),
\end{equation}
where
\begin{equation}
\label{eigenstate}\Psi_{\omega}(x)=\sqrt{\frac{\mu_\omega-m\omega^2x^2/2}{g}},\quad |x|\leq\sqrt{\frac{2\mu_\omega}{m\omega^2}},
\end{equation}
and the chemical potential $\mu_\omega$ is determined from the number of bosons $N$ through the normalization condition
\begin{equation}
\int|\psi(t,x)|^2dx=\int|\Psi_{\omega}(x)|^2dx=N.
\end{equation}
We find easily that
\begin{equation}
\label{energy}\mu_{\omega}=\left(\frac{9}{32}m\omega^2g^2N^2\right)^{1/3}.
\end{equation}

\begin{figure}[t]
\centering
\includegraphics[width=0.9\linewidth]{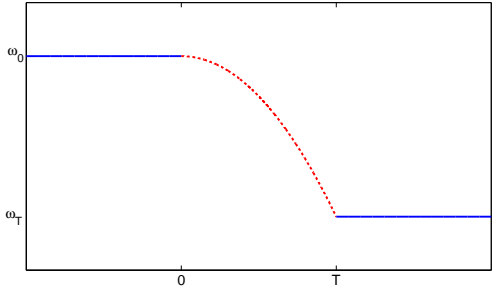}\caption{Time evolution of the
harmonic trap frequency.}%
\label{fig:frequency}%
\end{figure}

%\begin{figure}[t]
%\centering
%       \begin{tabular}{c}
%       \subfigure[$\ $]{
%               \label{fig:initial}
%               \includegraphics[width=.8\linewidth]{harmonic_0}} \\
%           \subfigure[$\ $]{
%               \label{fig:final}
%               \includegraphics[width=.8\linewidth]{harmonic_T}}
%       \end{tabular}
%\caption{Initial (panel a) and final (panel b) trapping parabolic potentials. For frictionless cooling, the populations of the corresponding energy levels should be the same, i.e. $|c_n(T)|^2=|c_n(0)|^2, n=0,1,2,\ldots$}
%\label{fig:potentials}
%\end{figure}

Consider now the case shown in Fig.\ \ref{fig:frequency}, where $\omega
(t)=\omega_{0}$ for $t\leq0$ and $\omega(t)=\omega_{T}<\omega_{0}$ for $t\geq
T$. For frictionless evolution, the path $\omega(t)$ between these two values should be
chosen such that if
\begin{equation}
\label{initial_condition}\psi(0,x)=\Psi_{\omega
_{0}}(x),\nonumber
\end{equation}
then
\begin{equation}
\label{frictionless_cooling}\psi(t,x)=e^{ia(t)}\Psi_{\omega_{T}%
}(x),\quad t\geq T,
\end{equation}
where $\alpha(t)$ is a global (independent of the spatial
coordinate $x$) phase factor. According to (\ref{eigenstate}) and (\ref{energy}), this evolution corresponds to the expansion of the condensate by the factor $(\omega_{0}/\omega_{T})^{2/3}$, see Fig. \ref{fig:expansion}.
Among all the paths $\omega(t)$ that result in (\ref{frictionless_cooling}),
we would like to find one that achieves frictionless expansion in minimum time
$T$. In the following we provide a sufficient condition on $\omega(t)$ for
frictionless expansion and we use it to formulate the corresponding time-optimal
control problem.

\begin{proposition}
\label{prop:fr_cooling} If $\omega(t)$, with $\omega(0)=\omega_{0}$ and
$\omega(t)=\omega(T)=\omega_{T}$ for $t\geq T$, is such that the
equation
\begin{equation}
\label{Ermakov}\ddot{b}(t)+\omega^{2}(t)b(t)=\frac{\omega_{0}^{2}}{b^{2}(t)}%
\end{equation}
has a solution $b(t)$, with $b(0)=1,\dot{b}(0)=0$ and $b(t)=b(T)=(\omega
_{0}/\omega_{T})^{2/3}$, $t\geq T$, then condition (\ref{frictionless_cooling})
for frictionless expansion is satisfied.
\end{proposition}
\begin{figure}[t]
\centering
\includegraphics[width=0.7\linewidth]{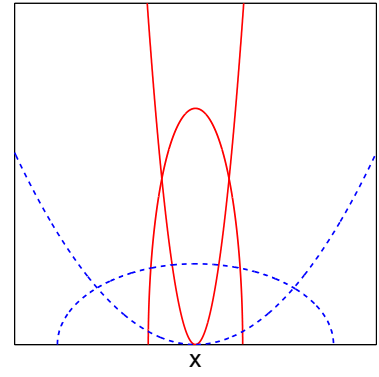}\caption{Schematic representation of the condensate expansion. With solid line we represent the initial harmonic potential (convex) and the corresponding wavefunction (concave), while with dashed line we depict the final potential (convex) and the expanded wavefunction (concave).}%
\label{fig:expansion}%
\end{figure}
\begin{proof}
The frequency variations in the trapping potential change the time and
distance scales and motivate the use of the following ``ansatz" \cite{Kagan96, Muga09}, in
(\ref{Schrodinger})
\begin{equation}
\label{ansatz}\psi(t,x)=\frac{1}{\sqrt{b(t)}}\exp{\left[
i\frac{mx^{2}}{2\hbar}\frac{\dot{b}(t)}{b(t)}\right]  }\phi(\tau,\chi),\nonumber
\end{equation}
where $\chi=x/b(t)$, $\tau=\tau(t)$ is a time rescaling, and the distance scale $b(t)$ satisfies (\ref{Ermakov})
and the accompanying boundary conditions. We obtain
\begin{equation}
\label{intermediate}i\hbar\frac{\partial\phi}{\partial\tau}\left(  \frac
{d\tau}{dt}b^2\right)  =\left[  -\frac{\hbar^{2}}{2m}\frac{\partial^{2}%
}{\partial\chi^{2}}+\frac{m(\ddot{b}+\omega^{2}b)b^{3}}{2}\chi^{2}+bg|\phi|^2\right]
\phi.
\end{equation}
If we choose the time scale $\tau(t)$ such that
\begin{equation}
\label{time_rescaling}\tau(t)=\int_{0}^{t}\frac{dt^{\prime}}{b(t^{\prime
})},
\end{equation}
then (\ref{intermediate}) becomes
\begin{equation}
\label{scaled}
i\hbar\frac{\partial\phi}{\partial\tau}b=\left(  -\frac{\hbar^{2}}{2m}%
\frac{\partial^{2}}{\partial\chi^{2}}+b\frac{m\omega_{0}^{2}}{2}\chi
^{2}+bg|\phi|^2\right)  \phi
\end{equation}
with the initial condition $\phi(0,\chi)=\Psi_{\omega_{0}}(\chi)$. In this equation the frequency is constant and we can apply the Thomas-Fermi approximation neglecting the kinetic energy term, provided that the scaling factor $b$ does not become very small (we will secure this by imposing appropriate conditions later),
\begin{equation}
i\hbar\frac{\partial\phi}{\partial\tau}=\left(\frac{m\omega_{0}^{2}}{2}\chi
^{2}+g|\phi|^2\right)  \phi.\nonumber
\end{equation}
So
$\phi(\tau,\chi)=e^{-i\mu_{\omega_{0}}\tau/\hbar}\Psi_{\omega_{0}}
(\chi)$ and
\begin{align}
\label{almost}\psi(t,x)&=\exp{\left[  i\frac{mx^{2}}{2\hbar}\frac
{\dot{b}(t)}{b(t)}\right]  }\times \exp\left[-i\frac{\mu_{\omega_{0}}\tau(t)}{\hbar}\right]\nonumber\\
& \times\frac{1}{\sqrt{b(t)}}\Psi_{\omega_{0}}(\frac{x}{b(t)})
\end{align}
We will show that for $t\geq T$, where $b(t)=(\omega_{0}/\omega_{T})^{2/3}$,
$\psi(t,x)$ has the desired form. We examine separately each of the three
terms in (\ref{almost}). Since $\dot{b}(t)=0$ in this time interval, the first
exponential is equal to unity. About the second exponential, observe from
(\ref{time_rescaling}) that
\begin{equation}
\tau(t)=\tau(T)+\left(\frac{\omega_{T}}{\omega_{0}}\right)^{2/3}(t-T),\nonumber
\end{equation}
since $b(t)=(\omega_{0}/\omega_{T})^{2/3}$, $t\geq T$. Also, from (\ref{energy})
we have $\mu_{\omega_{0}}/\mu_{\omega_{T}}=(\omega_{0}/\omega_{T})^{2/3}$. Thus
\begin{equation}
e^{-i\mu_{\omega_{0}}\tau(t)/\hbar}=e^{-i\mu_{\omega_{0}}\tau(T%
)/\hbar}e^{-i\mu_{\omega_{T}}(t-T)/\hbar}\nonumber
\end{equation}
The last term in (\ref{almost}) satisfies
\begin{equation}
\left(  \frac{\omega_{T}}{\omega_{0}}\right)  ^{1/3}\Psi_{\omega_{0}}\left[\left(\frac{\omega_{T}}{\omega_{0}}\right)^{2/3}x\right]=\Psi_{\omega_{T}}(x),\nonumber
\end{equation}
as it can be verified using (\ref{eigenstate}) and (\ref{energy}). Putting all these together we
see that $\psi(t,x)$ has the desired form for $t\geq T$.
\end{proof}

In order to find the path $\omega(t)$, $0\leq t\leq T$, that accomplishes
frictionless expansion in minimum time $T$, we express the problem using the
language of optimal control, incorporating possible restrictions on
$\omega(t)$ and $b(t)$ due, for example, to experimental limitations. If we set
\begin{equation}
x_{1}=b,\,x_{2}=\frac{\dot{b}}{\omega_{0}},\,u(t)=\frac{\omega^{2}(t)}%
{\omega_{0}^{2}},
\end{equation}
and rescale time according to $t_{\mbox{new}}=\omega_{0} t_{\mbox{old}}$, we
obtain the following system of first order differential equations, equivalent
to equation (\ref{Ermakov})
\begin{align}
\label{system1}\dot{x}_{1}  &  = x_{2},\\
\label{system2}
\dot{x}_{2}  &  = -ux_{1}+\frac{1}{x_{1}^{2}}.
\end{align}
If we set $\gamma=(\omega_{0}/\omega_{T})^{2/3}>1$, the time optimal problem
takes the following form

\newtheorem{problem}{problem} \begin{problem}\label{problem}
Find $-1\leq u(t) \leq 1$ with $u(0)=1, u(T)=1/\gamma^3$ such that starting from $(x_1(0),x_2(0))=(1,0)$, the above system reaches the final point $(x_1(T),x_2(T))=(\gamma,0), \gamma>1$ in minimum time $T$, under the constraint $x_2\geq0$.
\end{problem}

The boundary conditions on the state variables $(x_{1},x_{2})$ are equivalent
to those for $b,\dot{b}$, while the boundary conditions on the control
variable $u$ are equivalent to those for $\omega$, so the requirements of Proposition \ref{prop:fr_cooling} are satisfied. Note that
the possibility $\omega^{2}(t)<0$ (expulsive parabolic potential) for some
time intervals is permitted \cite{Muga09,Chen10}. The maximum allowed frequency is taken to be equal with the initial frequency $\omega_0^2$ \cite{Hoffmann_EPL11}. The lower bound is taken as $-\omega_0^2$ since the same harmonic potential can be made repelling or attractive by a phase shift when it is created by an optical lattice \cite{Hoffmann_EPL11}. The assumption $x_2\geq0$ implies that $\dot{b}\geq 0$, so the scale factor is an increasing function of time satisfying $b(t)\geq1$. This assumption guarantees that the condensate is continuously expanding and excludes breather-like solutions which might be difficult to implement experimentally. Note that the aforementioned condition holds for both time profiles of $b(t)$ used in \cite{Muga09}. Specifically, for the polynomial form $b_p(s)=(\gamma-1)(6s^5-15s^4+10s^3)+1$, where $s=t/T$, it is $\dot{b}_p(s)=30(\gamma-1)s^2(s-1)^2\geq 0$, while for the exponential of polynomial form $b_e(s)=\gamma^{\,6s^5-15s^4+10s^3}$ it is $\dot{b}_e(s)=30\ln{\gamma} \,s^2(s-1)^2b_e(s)\geq 0$. Now observe that the system (\ref{system1}) and (\ref{system2}) describes the one-dimensional Newtonian motion of a unit-mass particle,
with position coordinate $x_{1}$ and velocity $x_{2}$. The acceleration
(force) acting on the particle is $-ux_{1}+1/x_{1}^{2}$. From this point of view, the breather-like solutions correspond to retrograde motion \cite{Rezek11}. The full benefits from this kind of motion arise when the particle moves close to the strong repulsive potential at $x_1=0$ \cite{Stefanatos10}. But when $b\rightarrow 0$ the Thomas-Fermi approximation in (\ref{scaled}) becomes questionable. For all these reasons we will not consider breather-like solutions here and we defer the study of retrograde motion to a future publication.

%If we omit the boundary conditions on $u(t)$ and solve the corresponding
%problem, the minimum time that we find is a lower bound of the minimum time
%for the full Problem \ref{problem}, where these conditions are on. This bound
%is obtained with instantaneous jumps of the control at the initial and final
%points. In the following we solve this relaxed problem

In the next section we solve the following optimal control problem

\begin{problem}\label{problem1}
Find $-1\leq u(t)\leq 1$ such that starting from $(x_1(0),x_2(0))=(1,0)$, the system above reaches the final point $(x_1(T),x_2(T))=(\gamma,0), \gamma>1$ in minimum time $T$, under the constraint $x_2\geq0$.
\end{problem}

In both problems the class of admissible controls formally are Lebesgue measurable functions that take values in the control set $[-1,1]$ almost everywhere. However, as we shall see, optimal controls are piecewise continuous, in fact bang-bang. The optimal control found for problem \ref{problem1} is also optimal for problem \ref{problem}, with the addition of instantaneous jumps at the initial and final points, so that the boundary conditions $u(0)=1$ and $u(T)=1/\gamma^3$ are satisfied. Note that in connection with Fig. \ref{fig:frequency}, a natural way to think about these conditions is that $u(t)=1$ for $t\leq 0$ and $u(t)=1/\gamma^3$ for $t\geq T$; in the interval $(0,T)$ we pick the control that achieves the desired transfer in minimum time. This approach is similar to that used in our recently published work \cite{Stefanatos11, Chen11}, as well as in \cite{Hoffmann_EPL11, Salamon09}.

%%%%%%%%%%%%%%%%%%%%%%%%%%%%%%%%%%%%%%%%%%%%%%%%%%%%%%%%%%%%%%%%%%%%%%%%%%%%%%%%

\section{OPTIMAL SOLUTION}

The system described by (\ref{system1}), (\ref{system2}) can be expressed
in compact form as
\begin{equation}
\dot{x}=f(x)+ug(x), \label{affine}
\end{equation}
where the vector fields are given by
\begin{equation}
f=\left(
\begin{array}
[c]{c}
x_{2}\\
1/x_{1}^{2}
\end{array}
\right)  ,\,\,g=\left(
\begin{array}
[c]{c}
0\\
-x_{1}
\end{array}
\right),
\end{equation}
$x\in D=\{(x_{1},x_{2})\in\mathbb{R}^{2}:x_{1}>0,x_2\geq0\}$ and $u\in
U=[-1,1]$. Admissible controls are Lebesgue measurable functions that
take values in the control set $U$. While the bound in $x_2$ is imposed to exclude retrograde motion, the bound in $x_1$ results form the initial condition and the system dynamics. Specifically, observe that starting with any positive initial condition $x_{1}(0)>0$, and
using any admissible control $u$, as $x_{1}\rightarrow0^{+}$, the
\textquotedblleft repulsive force" $1/x_{1}^{2}$ leads to an increase in
$x_{1}$ that will keep $x_{1}$ positive (as long as the solutions exist). Given an admissible control $u$ defined
over an interval $[0,T]$, the solution $x$ of the system (\ref{affine})
corresponding to the control $u$ is called the corresponding trajectory and we
call the pair $(x,u)$ a controlled trajectory.

For a constant $\lambda_{0}$ and a row vector $\lambda=(\lambda_{1}%
,\lambda_{2})\in\left(  \mathbb{R}^{2}\right)  ^{\ast}$ define the control
Hamiltonian as%
\[
H=H(\lambda_{0},\lambda,x,u)=\lambda_{0}+\langle\lambda,f(x)+ug(x)\rangle.
\]
Then, Pontryagin's Maximum Principle \cite{Pontryagin}
provides the following necessary conditions for optimality:

\begin{theorem}
\label{prop:max_principle}
Let $(x_{\ast}(t),u_{\ast}(t))$
be a time-optimal controlled trajectory that transfers the initial condition
$x(0)=x_{0}$ into the terminal state $x(T)=x_{T}$. Then it is a necessary
condition for optimality that there exists a constant $\lambda_{0}\leq0$ and
nonzero, absolutely continuous row vector function $\lambda(t)$ such that:

\begin{enumerate}
\item $\lambda$ satisfies the so-called adjoint equation%
\begin{align}
\dot{\lambda}(t) &=-\frac{\partial H}{\partial x}(\lambda_{0},\lambda
(t),x_{\ast}(t),u_{\ast}(t)) \nonumber\\
&=-\left\langle \lambda(t),Df(x_{\ast}(t))+u_{\ast
}(t)Dg(x_{\ast}(t))\right\rangle
\end{align}

\item For $0\leq t\leq T$ the function $u\mapsto H(\lambda_{0}%
,\lambda(t),x_{\ast}(t),u)$ attains its maximum\ over the control set $U$ at
$u=u_{\ast}(t)$.

\item $H(\lambda_{0},\lambda(t),x_{\ast}(t),u_{\ast}(t))\equiv0$.
\end{enumerate}
\end{theorem}

We call a controlled trajectory $(x,u)$ for which there exist multipliers
$\lambda_{0}$ and $\lambda(t)$ such that these conditions are satisfied an
extremal. Extremals for which $\lambda_{0}=0$ are called abnormal. If
$\lambda_{0}<0$, then without loss of generality we may rescale the $\lambda
$'s and set $\lambda_{0}=-1$. Such an extremal is called normal.
%Abnormal
%extremals typically correspond to some degeneracies in the structure of the
%optimal solution (often the value function is no longer differentiable along
%these paths), but they cannot be excluded a priori for time-optimal control
%problems. For example, the solution to the time-optimal control problem to the
%origin for the harmonic oscillator, a simple text book example, is largely
%characterised by two optimal abnormal controlled trajectories.

%\newtheorem{definition}{Definition}
\begin{definition}
We denote the vector fields corresponding to the constant controls
$u=-1$ and $u=+1$ by $X=f-g$ and $Y=f+g$, respectively, and call
the corresponding trajectories the \emph{bang} $X$- and $Y$-trajectories. A \emph{bang-bang} trajectory is a
finite concatenation of $X$- and $Y$-trajectories. A concatenation of an
$X$-trajectory followed by a $Y$-trajectory is denoted by $XY$ while the
concatenation in the inverse order is denoted by $YX$.
\end{definition}

\begin{proposition}
For Problem \ref{problem1} the optimal trajectory has the bang-bang form $XY$.
\end{proposition}

\begin{proof}
For the system (\ref{system1}), (\ref{system2}) we have
\begin{equation}
H(\lambda_{0},\lambda,x,u)=\lambda_{0}+\lambda_{1}x_{2}+\lambda_{2}\left(  \frac{1}%
{x_{1}^{2}}-x_{1}u\right)  ,\label{hamiltonian}%
\end{equation}
and thus
\begin{equation}
\dot{\lambda}=\lambda\left(
\begin{array}
[c]{cc}%
0 & -1\\
u+2/x_{1}^{3} & 0
\end{array}
\right) \label{adjoint}%
\end{equation}

Observe that $H$ is a linear function of the bounded control variable $u$. The
coefficient at $u$ in $H$ is $-\lambda_{2}x_{1}$ and, since $x_{1}>0$, its
sign is determined by $\Phi=-\lambda_{2}$, the so-called \emph{switching
function}. According to the maximum principle, point 2 above, the optimal
control is given by $u=-1$ if $\Phi<0$ and by $u=1$ if $\Phi>0$. The
maximum principle provides a priori no information about the control at times
$t$ when the switching function $\Phi$ vanishes. However, if $\Phi(t)=0$ and
$\dot{\Phi}(t)\neq0$, then at time $t$ the control switches between its
boundary values and we call this a bang-bang switch. If $\Phi$ were to vanish
identically over some open time interval $I$ the corresponding control is
called \emph{singular}. Now observe that, whenever the switching function $\Phi(t)=-\lambda_{2}(t)$ vanishes at some
time $t$, then it follows from the non-triviality of the multiplier
$\lambda(t)$ that its derivative $\dot{\Phi}(t)=-\dot{\lambda}_{2}%
(t)=\lambda_{1}(t)$ is non-zero. Hence the switching function changes sign and
there is a bang-bang switch at time $t$. Thus optimal controls alternate between the boundary values $u=-1$ and
$u=1$ of the control set.

We next show how the restriction $x_2\geq 0$ limits the number and the type of switchings.
For $u=1$ the initial point $(1,0)$ is an equilibrium point for system (\ref{system1}), (\ref{system2}) so we should start with $u=-1$, moving the system along an $X$-trajectory. Observe from (\ref{system2}) that for $u=-1$ it is $\dot{x}_{2}>0$ so $x_2>0$, and a switching to a $Y$-trajectory at a point with $x_{2}(t)>0$ is necessary in order to reach the target point $(\gamma,0)$. Then, since $\lambda_{2}(t)=0$ at the switching point, it follows from $H=0$ that $\lambda
_{1}(t)x_{2}(t)=-\lambda_{0}$. Since $\lambda_{1}(t)=\dot{\Phi}(t)>0$, it is $\lambda_{0}<0$. We henceforth only
consider normal trajectories and set $\lambda_{0}=-1$.
Then, $H=0$ implies that for any switching time $t'$ we must have $\lambda
_{1}(t')x_{2}(t')=1$. Since $x_{2}(t')\geq 0$ it is $\dot{\Phi}
(t')=\lambda_{1}(t')>0$ and obviously the initial switching from $X$ to $Y$ is unique.
\end{proof}

\begin{figure}[t]
\centering
\includegraphics[width=0.8\linewidth]{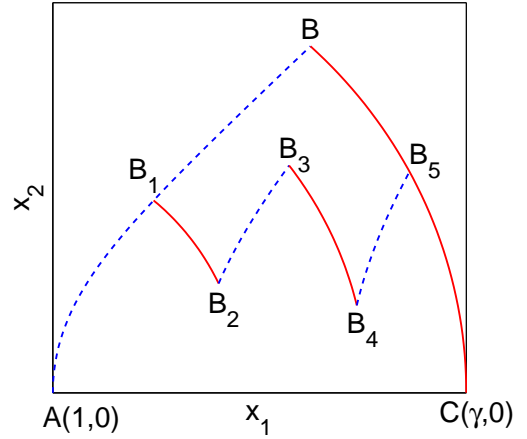}\caption{One- and multi-switchings trajectories. Solid line corresponds to $u=1$, dashed line to $u=-1$.}
\label{fig:comparison}
\end{figure}

An intuitive way to understand why the one-switching trajectory is the fastest is to compare it with a trajectory with more switchings, like that shown in Fig. \ref{fig:comparison}, and use the particle model mentioned above. In both trajectories the particle traverses the same distance $x_1$ but along the trajectory with one switching its speed $\dot{x}_1=x_2$ is always higher.

There is an elegant way to show the time-optimality of the one-switching trajectory, at least when compared to trajectories with multiple switchings at points with $x_2>0$, using the one-form $\Omega$ introduced in \cite{Su1} and defined by  $\Omega(f)=1$ and $\Omega(g)=0$. Note that for $x_2\neq 0$ the vectors $f,g$ are linearly independent in $D$ so the one-form can be defined in this two-dimensional manifold by its action on these fields. Now consider the parts where the two trajectories shown in Fig. \ref{fig:comparison} differ and denote them by $c$ (part $B_1BB_5$) and $c'$ (part $B_1B_2B_3B_4B_5$).
The infinitesimal displacement along each trajectory is $dx=(f+ug)dt$, so $dt=\Omega(dx)$. The necessary time to travel $c$ is $t=\int dt=\int_c \Omega$ and the corresponding time for $c'$ is $t'=\int dt=\int_{c'} \Omega$. So $t'-t=\int_{c'} \Omega-\int_c \Omega=\oint_{c^{-1}*c'} \Omega$, where $c^{-1}$ is $c$ run backwards. Since $c^{-1}*c'$ is oriented counterclockwise, by Stokes theorem we have $t'-t=\int_R d\Omega$, where $R$ is the region enclosed by $c^{-1}*c'$. From the definition of $\Omega$ we can find its expression in cartesian coordinates $\Omega=(1/x_2)dx_1$ and from this the exterior derivative $d\Omega=(1/x_2^2)dx_1\wedge dx_2$ \cite{Schutz80}. Thus
\begin{equation}
t'-t=\int_R d\Omega=\int_R\frac{1}{x_2^2}dx_1dx_2>0\nonumber
\end{equation}
and obviously the trajectory with only one switching is faster.

We now move to calculate the minimum time $T$ necessary to reach the final point $(\gamma,0)$. A first integral of the motion along the $X$-trajectory starting from $(1,0)$ is
\begin{equation}
x_{2}^{2}-x_{1}^{2}+\frac{2}{x_{1}}=1,\label{first_integral_1}
\end{equation}
while a first integral of the motion along the $Y$-trajectory ending in $(\gamma,0)$ is
\begin{equation}
x_{2}^{2}+x_{1}^{2}+\frac{2}{x_{1}}=\gamma^2+\frac{2}{\gamma},\label{first_integral_2}
\end{equation}
The total time $T$ is given by the sum of the times spent on each trajectory segment
\begin{equation}
\label{total_time}
T=T_1+T_2.
\end{equation}
Since $\dot{x}_1=x_2>0$ for both cases, we can easily find from (\ref{first_integral_1})
\begin{equation}
\label{integral_1}
T_1=\int_1^\beta\sqrt{\frac{x_1}{(x_1-1)(x_1^2+x_1+2)}}\,dx_1
\end{equation}
and from (\ref{first_integral_2})
\begin{equation}
\label{integral_2}
T_2=\int_\beta^\gamma\sqrt{\frac{x_1}{(\gamma-x_1)(x_1^2+\gamma x_1-2/\gamma)}}\,dx_1,
\end{equation}
where $B(\beta,\delta)$ is the common point of the $X$ and $Y$ segments shown in Fig. \ref{fig:comparison}, satisfying (\ref{first_integral_1}) and (\ref{first_integral_2}). It is
\begin{equation}
\label{beta}
\beta=\sqrt{\frac{\gamma^3-\gamma+2}{2\gamma}}.
\end{equation}
We express the times $T_1, T_2$ in terms of elliptic integrals and elementary functions.
We start from (\ref{integral_1}) by rewriting the integrand in the convenient form (note that $x_1\neq 0$)
\begin{equation}
\label{integral_1_convenient}
T_1=\int_1^\beta\frac{x_1}{\sqrt{(x_1^2-x_1)(x_1^2+x_1+2)}}\,dx_1.
\end{equation}
Following the procedure described in \cite{Abramowitz65}, we can bring the above integral in the form
\begin{equation}
\label{integral_1_transformed}
T_1=\frac{1}{\sqrt{B_1B_2}(b-a)}\int_q^y\frac{b-ay}{(1-y)\sqrt{(y^2+p^2)(y^2-q^2)}}\,dy,
\end{equation}
where
\begin{equation}
\label{transformation_y}
y=\frac{x_1-b}{x_1-a}
\end{equation}
and
\begin{equation}
a=-\frac{\sqrt{2}-1}{3-2\sqrt{2}},\quad b=\frac{1+\sqrt{2}}{3+2\sqrt{2}},
\end{equation}
\begin{equation}
B_1=\frac{1+2\sqrt{2}}{4\sqrt{2}},\quad B_2=\frac{3+2\sqrt{2}}{4\sqrt{2}},
\end{equation}
\begin{equation}
p^2=\frac{2\sqrt{2}-1}{2\sqrt{2}+1},\quad q^2=\frac{3-2\sqrt{2}}{3+2\sqrt{2}}.
\end{equation}
The integral in (\ref{integral_1_transformed}) can be decomposed as
\begin{equation}
T_1=\frac{1}{\sqrt{B_1B_2}}\left(\frac{I_1}{b-a}+I_2\right),
\end{equation}
where
\begin{equation}
\label{elliptic_part}
I_1=\int_q^y\frac{b-ay^2}{(1-y^2)\sqrt{(y^2+p^2)(y^2-q^2)}}\,dy
\end{equation}
and
\begin{equation}
\label{elementary_part}
I_2=\int_q^y\frac{y}{(1-y^2)\sqrt{(y^2+p^2)(y^2-q^2)}}\,dy.
\end{equation}
The substitution
\begin{equation}
\label{transformation_z}
z=\sqrt{1-\frac{q^2}{y^2}}
\end{equation}
allows us to express $I_1$ in terms of elliptic integrals
\begin{equation}
\label{I_1}
I_1=\frac{1}{\sqrt{p^2+q^2}}\left[bF(z,m)+\frac{(b-a)q^2}{1-q^2}\Pi(n,z,m)\right],
\end{equation}
where
\begin{equation}
\label{Elliptic_1}
F(z,m)=\int_0^z\frac{dx}{\sqrt{(1-x^2)(1-mx^2)}}
\end{equation}
and
\begin{equation}
\label{Elliptic_3}
\Pi(n,z,m)=\int_0^z\frac{dx}{(1-ny^2)\sqrt{(1-x^2)(1-mx^2)}}
\end{equation}
are the incomplete elliptic integrals of the first and third kind, respectively, in Jacobi's form, while
\begin{equation}
m=\frac{p^2}{p^2+q^2},\quad n=\frac{1}{1-q^2}.
\end{equation}
On the other hand, the substitution
\begin{equation}
\label{transformation_w}
w=\frac{\frac{2c}{1-y^2}-d}{p^2+q^2},
\end{equation}
where
\begin{equation}
c=(1+p^2)(1-q^2),\quad d=2+p^2-q^2,
\end{equation}
gives
\begin{equation}
I_2=\frac{1}{2\sqrt{c}}\int_1^w\frac{dw}{\sqrt{w^2-1}}=\frac{1}{2\sqrt{c}}\ln|w+\sqrt{w^2-1}|
\end{equation}
Using all the above, the time $T_1$ spent on the $X$ segment of the trajectory can be calculated as a function of $\gamma$. We move on to calculate $T_2$. Integral (\ref{integral_2}) can be written as
\begin{equation}
\label{integral_2_transformed}
T_2=\int_\beta^\gamma\sqrt{\frac{x_1}{(\gamma-x_1)(x_1+\epsilon)(x_1-\zeta)}}\,dx_1,
\end{equation}
where
\begin{equation}
\epsilon=\frac{\gamma+\sqrt{\gamma^2+8/\gamma}}{2},\quad \zeta=\frac{-\gamma+\sqrt{\gamma^2+8/\gamma}}{2}.
\end{equation}
Since $-\epsilon<0<\zeta<\gamma$ and $\zeta<\beta<\gamma$ (recall that $\gamma>1$), we find from \cite{Prudnikov86} that
\begin{equation}
\label{integral_2_elliptic}
T_2=\frac{2}{\sqrt{\gamma(\epsilon+\zeta)}}\left[(\gamma+\epsilon)\Pi(\nu,x,\mu)-\epsilon F(x,\mu)\right],
\end{equation}
where
\begin{equation}
\mu=\frac{\epsilon(\gamma-\zeta)}{\gamma(\epsilon+\zeta)},\quad \nu=-\frac{\gamma-\zeta}{\epsilon+\zeta},
\end{equation}
and
\begin{equation}
x=\sqrt{\frac{(\epsilon+\zeta)(\gamma-\beta)}{(\gamma-\zeta)(\beta+\epsilon)}}.
\end{equation}
Note that $\beta$ is given in (\ref{beta}) as a function of $\gamma$. In Fig. \ref{fig:time} we plot the total expansion time $T=T_1+T_2$ with respect to $\gamma$.

We finally show that for large $\gamma$ time $T$ grows logarithmically, which seems to be a universal characteristic for processes involving shortcuts to adiabaticity \cite{Stefanatos10,Chen10b,Hoffmann_EPL11}. Instead of taking the limit $\gamma\rightarrow\infty$ in the above expressions, we follow an easier path. Note that for large $x_1$ the system equations (\ref{system1}) and (\ref{system2}) take the simple form
\begin{align}
\label{limit_system1}\dot{x}_{1}  &  = x_{2},\\
\label{limit_system2}
\dot{x}_{2}  &  = -ux_{1}.
\end{align}
For $u=-1$ this system corresponds to a harmonic oscillator with angular frequency $\omega=1$. The switching point $B$ in Fig. \ref{fig:comparison} tends to $B\rightarrow(\gamma/\sqrt{2},\gamma/\sqrt{2})$, so it is close to the line through the origin which makes a $\pi/4$ angle with the $x_1$-axis. The necessary time to reach the final point $(\gamma,0)$ from $B$ approaches $T_2\rightarrow\pi/(4\omega)=\pi/4$. For $u=1$ we find from (\ref{limit_system1}) and (\ref{limit_system2}) that for large $x_1$ and $t$ it is $x_1\sim e^t$, so $T_1\sim\ln{\gamma}$. Thus, the dependence of total time for large $\gamma$ is $T\sim\ln{\gamma}$.

\begin{figure}[t]
\centering
\includegraphics[width=0.8\linewidth]{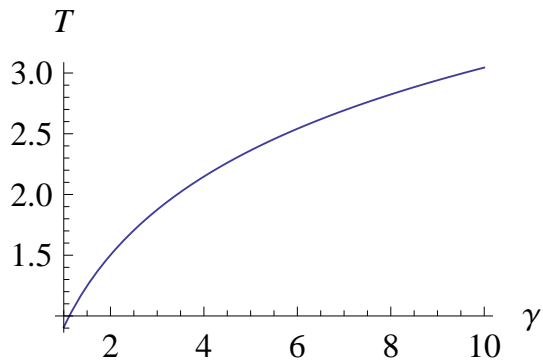}\caption{Total time for the adiabatic-like expansion as a function of the expansion factor.}
\label{fig:time}
\end{figure}

%%%%%%%%%%%%%%%%%%%%%%%%%%%%%%%%%%%%%%%%%%%%%%%%%%%%%%%%%%%%%%%%%%%%%%%%%%%%%%%%
\section{CONCLUSION AND FUTURE WORK}

In this paper, we combined optimal control theory with shortcut to adiabaticity to derive the optimal time variation of a confining harmonic potential which results in a fast adiabatic-like expansion of a trapped one-dimensional BEC. This fast frictionless expansion is an important task for applications where controlling BEC in an efficient way is crucial, for example atom interferometry. In our future work we would like to relax some of the assumptions used in this article, for example to permit breather-like solutions for the condensate evolution.

\end{document}